\documentclass[10pt]{amsart}
\usepackage{enumerate,amsmath,amssymb,latexsym,
amsfonts, amsthm, amscd}

\theoremstyle{plain}

\newtheorem{theorem}{Theorem}

\numberwithin{equation}{section}

\newcommand{\ra}{\rightarrow}

\newcommand{\mdot}{\,\begin{picture}(-1,-1)(-1,-1)\circle*{2}\end{picture}\ }

\addtocounter{section}{-1}

\begin{document}

\title {Painlev\'e IV and a third-order viewpoint}

\date{}

\author[P.L. Robinson]{P.L. Robinson}

\address{Department of Mathematics \\ University of Florida \\ Gainesville FL 32611  USA }

\email[]{paulr@ufl.edu}

 \keywords{}

\begin{abstract}

We take a third-order approach to the fourth Painlev\'e equation and indicate the value of such an approach to other second-order ODEs in the Painlev\'e-Gambier list of 50.  

\end{abstract}

\maketitle

\medbreak

\section*{Introduction} 

\medbreak 

Among the full list of six, the fourth Painlev\'e equation ${\bf PIV}$ is the most complicated for which all solutions are meromorphic throughout the complex plane. ${\bf PIV}$ has the form 
$$\overset{\mdot \mdot}{w} = F(z, w, \overset{\mdot}{w})$$
where the right side is rational in all three variables: it fails to be polynomial in all variables only in having $w$ in the denominator; for its precise form, see below. On the one hand, this means that the standard (local) existence-uniqueness theorem for second-order ODEs applies to ${\bf PIV}$ with initial data in which $w(a) \ne 0$ and $\overset{\mdot}{w}(a)$ are specified; on the other hand, such a standard theorem does not apply to ${\bf PIV}$ with initial data involving $w(a) = 0$. Further differentiation produces a third-order ODE of the form 
$$\overset{\mdot \mdot \mdot}{w} = G(z, w, \overset{\mdot}{w}, \overset{\mdot \mdot}{w})$$
where the right side is polynomial in all variables (and $\overset{\mdot \mdot}{w}$ is absent). A standard (local) existence-uniqueness theorem for third-order ODEs applies to this equation: locally, there exists a unique solution $w$ for which $w(a), \; \overset{\mdot}{w}(a)$ and $ \overset{\mdot \mdot}{w}(a)$ take specified values. This circumstance has its consequences for ${\bf PIV}$, some of which we address. One reason for the simplification that arises upon passage to third order is that derivatives enter ${\bf PIV}$ only in the combination $\overset{\mdot \mdot}{w} - \overset{\mdot}{w}^2/2w$. This precise combination appears in a further dozen of the 50 canonical forms that stem from the analysis of Painlev\'e and Gambier as listed by Ince in [1]; the third-order approach may be profitably considered there also, as we illustrate in a couple of cases. 

\medbreak 

\section*{Painlev\'e IV}

\medbreak 

The precise form of the fourth Painlev\'e equation in the literature varies as to the naming of its parameters; we find it convenient to adopt the form 
\begin{equation} \label{PIV} 
\overset{\mdot \mdot}{w} = \frac{\overset{\mdot}{w}^2}{2 w} + \frac{3}{2} w^3 + 4 z w^2 + 2 (z^2 - \alpha) w - \frac{\beta^2}{2 w}. \tag{{$\bf PIV$}}
\end{equation}
This is the form taken in [1] when the equation appears as ${\bf XXXI}$ in the list of 50 canonical forms; upon the extraction of ${\bf PIV}$ as fourth in the list of six Painlev\'e equations, the parameter $\beta^2$ is relabelled as $- 2 \beta$. The ratio $(\overset{\mdot}{w}^2 - \beta^2)/2w$ on the right side of ${\bf PIV}$ is to be understood as a limit when appropriate: thus, if $w$ vanishes at $a$ then its first derivative at $a$ satisfies $\overset{\mdot}{w}(a)^2 = \beta^2$ or $\overset{\mdot}{w}(a) = \pm \beta$; this convenience is the reason behind our choice of form.  

\medbreak 

It is natural to view ${\bf PIV}$ as a complex equation: when we do so, the differentiability of a solution entails its analyticity where defined; analyticity at an (isolated) zero involves the Riemann continuation theorem. We may instead view ${\bf PIV}$ as a real equation: where defined, each solution is then plainly thrice-differentiable (and better) away from its zeros; at each zero, the limit understanding of the ratio on the right side of the equation renders the second derivative continuous, higher differentiability following by the mean value theorem. 

\medbreak 

Calculation of the third derivative is most conveniently effected after rearrangement, thus 
$$2 w \overset{\mdot \mdot}{w} = \overset{\mdot}{w}^2 + 3 w^4 + 8 z w^3 + 4 (z^2 - \alpha) w^2 - \beta^2.$$
Upon differentiation away from zeros, $2 \overset{\mdot}{w} \overset{\mdot \mdot}{w}$ cancels from each side to yield 
$$2 w \overset{\mdot \mdot \mdot}{w} = 12 w^3 \overset{\mdot}{w} + 24 z w^2 \overset{\mdot}{w} + 8 w^3 + 8 (z^2 - \alpha) w \overset{\mdot}{w} + 8 z w^2$$
whence 
$$\overset{\mdot \mdot \mdot}{w} = \{ 6 w^2 + 12 z w + 4 (z^2 - \alpha) \} \overset{\mdot}{w} + 4 (w + z) w.$$
Now let $a$ be an isolated zero of $w$: if $z \ra a$ then $w(z) \ra w(a) = 0$ and $\overset{\mdot}{w}(z) \ra \overset{\mdot}{w}(a) = \pm \beta$ so that $\overset{\mdot \mdot \mdot}{w}(z) \ra 4 (a^2 - \alpha) \overset{\mdot}{w}(a)$; it follows that $w$ is thrice-differentiable at $a$ with 
$$\overset{\mdot \mdot \mdot}{w}(a) = 4 (a^2 - \alpha) \overset{\mdot}{w}(a).$$
In other words, $w$ continues to satisfy the foregoing third-order equation at isolated zeros. 

\medbreak 

\begin{theorem} \label{third} 
If $w$ is a solution to ${\bf PIV}$ then $w$ satisfies the third-order ODE 
\begin{equation} \label{PIV'} 
\overset{\mdot \mdot \mdot}{w} = \{ 6 w^2 + 12 z w + 4 (z^2 - \alpha) \} \overset{\mdot}{w} + 4 (w + z) w. \tag{{$\bf PIV'$}}
\end{equation}
\end{theorem} 

\begin{proof} 
The proof precedes the statement, in which we assume the zeros of the solution to be isolated. 
\end{proof}  

\medbreak 

As mentioned in the Introduction, the form of ${\bf PIV'}$ guarantees that it satisfies the standard (local) existence-uniqueness theorem appropriate to third-order ODEs: locally, there exists a unique solution $w$ to ${\bf PIV}$ for which $w(a), \; \overset{\mdot}{w}(a)$ and $ \overset{\mdot \mdot}{w}(a)$ assume arbitrarily specified values. As also mentioned in the Introduction, the standard existence-uniqueness theorem appropriate to second-order ODEs is not applicable to ${\bf PIV}$ itself: in general, there is no solution to ${\bf PIV}$ with $w(a) = 0$ unless $\overset{\mdot}{w}(a) = \pm \beta$; in case $\beta = 0$, if $w$ satisfies $w(a) = 0$ then $\overset{\mdot}{w}(a) = 0$ and the standard second-order theorem would force $w$ to vanish identically near $a$ (which it need not do). At the risk of repetition, it is of course the case that not all solutions to ${\bf PIV'}$ satisfy ${\bf PIV}$: regarding those that do, if $w(a) \ne 0$ and $\overset{\mdot}{w}(a)$ are specified then $\overset{\mdot \mdot}{w}(a)$ must be as determined from ${\bf PIV}$, while if $w(a) = 0$ then $\overset{\mdot}{w}(a) = \pm \beta$ necessarily.

\medbreak 

Our passage to the third-order equation ${\bf PIV'}$ has applications to the Painlev\'e equation ${\bf PIV}$ itself. We first consider the special case in which $\beta$ is zero. 

\medbreak 

\begin{theorem} \label{zero}
Let $w$ be a solution to ${\bf PIV}$ in case $\beta = 0$. If $w$ has an isolated zero at $a$ then $\overset{\mdot \mdot}{w}(a) \ne 0$. 
\end{theorem} 

\begin{proof} 
Note that $w$ is a solution to ${\bf PIV'}$ by Theorem \ref{third}; note also that the vanishing of $w(a)$ implies the vanishing of $\overset{\mdot}{w}(a)$. If $\overset{\mdot \mdot}{w}(a)$ vanishes too, then the standard existence-uniqueness theorem for the third-order equation ${\bf PIV'}$ forces $w$ to vanish near $a$; consequently, the zero at $a$ cannot be isolated.  
\end{proof} 

\medbreak 

Let us now take both parameters $\alpha$ and $\beta$ to be zero, and write ${\bf PIV}_0$ for the resulting version of the fourth Painlev\'e equation, thus: 
$$\overset{\mdot \mdot}{w} = \frac{\overset{\mdot}{w}^2}{2 w} + \frac{3}{2} w^3 + 4 z w^2 + 2 z^2 w.$$
View this as a real equation, with real solutions. Let the real solution $w$ to real ${\bf PIV}_0$ have $a$ as an isolated zero. By Theorem \ref{zero} it follows that either $\overset{\mdot \mdot}{w}(a) > 0$ or $\overset{\mdot \mdot}{w}(a) < 0$: in the former case, $\overset{\mdot}{w}(t)$ passes from strictly negative to strictly positive as $t$ increases through $a$;  in the latter case, $\overset{\mdot}{w}(t)$ passes from strictly positive to strictly negative as $t$ increases through $a$. Now restrict to an open interval $I \ni a$ in which $a$ is the unique zero of $w$. The rule 
 \begin{equation*}
    f(t)=
    \begin{cases}
      - \sqrt{w(t)} & \text{if}\ I \ni t \leqslant a \\
      + \sqrt{w(t)} & \text{if} \ I \ni t \geqslant a 
    \end{cases}
  \end{equation*}
then defines on $I$ a square-root $f$ of $w$ that satisfies the second-order ODE 
$$4 \overset{\mdot \mdot}{f} = f (3 f^2 + 2 t)(f^2 + 2 t).$$
Observe that, unlike ${\bf PIV}$ itself, this differential equation for $f$ obeys the standard existence-uniqueness theorem for second-order ODEs.  For more detail on this, see particularly Theorem 5 in [2]. Incidentally, notice that Theorem \ref{zero} of the present paper justifies the discussion that surrounds Theorem 8 (and thereby supports Theorem 5) of [2] without invoking the (difficult) meromorphicity of solutions to ${\bf PIV}$. 

\medbreak 

\section*{Twelve more ODEs}

\medbreak 

Passage to a third-order ODE has benefits not only for the fourth Painlev\'e equation but also for other second-order ODEs among the 50 canonical forms that result from the Painlev\'e-Gambier classification as listed in [1]. The third-order approach succeeds in part because derivatives enter ${\bf PIV}$ in the precise combination 
$$\overset{\mdot \mdot}{w} - \frac{\overset{\mdot}{w}^2}{2 w} = \frac{2 w \overset{\mdot \mdot}{w} - \overset{\mdot}{w}^2}{2 w}$$
and 
$$\big(2 w \overset{\mdot \mdot}{w} - \overset{\mdot}{w}^2\big)^{\mdot} = 2 w \: \overset{\mdot \mdot \mdot}{w}.$$
Now, a further twelve equations in the list of 50 feature derivatives in just this combination; accordingly, analogous passage to the associated third-order ODE may be contemplated in these cases, too. We have already taken such an approach in [3], where we analyzed the relationship between the (homogeneous) second Painlev\'e equation and equation ${\bf XX}$ from the list of 50. In the present section, we focus primarily on the practical utility of passage to third order, which we illustrate by a couple more cases.  

\medbreak 

The thirty-second equation in the list of 50 has the following form: 
\begin{equation} \label{XXXII} 
\overset{\mdot \mdot}{w} =  \frac{\overset{\mdot}{w}^2}{2 w} - \frac{1}{2 w} =  \frac{\overset{\mdot}{w}^2 - 1}{2 w}. \tag{{$\bf XXXII$}}
\end{equation}
It is recorded in [1] that the substitution $w = u^2$ engenders a first integral (though the recorded first integral contains a minor misprint). In fact 
$$2 \overset{\mdot}{u}^2 + 2 u \overset{\mdot \mdot}{u} = \overset{\mdot \mdot}{w} =  \frac{\overset{\mdot}{w}^2}{2 w} - \frac{1}{2 w} = \frac{(2 u \overset{\mdot}{u})^2}{2 u^2} - \frac{1}{2 u^2} = 2 \overset{\mdot}{u}^2 - \frac{1}{2 u^2}$$
so that 
$$2 \overset{\mdot}{u} \overset{\mdot \mdot} {u} = - \frac{\overset{\mdot}{u}}{2 u^3}$$
which leads immediately to the first integral 
$$\overset{\mdot}{u}^2 = K + \frac{1}{4 u^2}.$$
It is certainly possible to integrate further, so as to determine $u$ and thereby determine $w$. However, it is instructive to pass directly from ${\bf XXXII}$ to the associated third-order equation. Thus: rearrange to obtain 
$$2 w \overset{\mdot \mdot}{w} = \overset{\mdot}{w}^2 - 1$$
and differentiate to deduce  
$$2 w \: \overset{\mdot \mdot \mdot}{w} = 0$$
wherefrom it is immediate that $w$ is (at most) a quadratic  
$$w = a z^2 + b z + c;$$
all that remains is to filter this solution to the third-order equation through ${\bf XXXII}$ itself, which yields the constraint $b^2 - 4 a c = 1$. Passage to third-order has furnished a surprisingly direct route to the solution of this particular second-order equation. 

\medbreak 

The $`m = 2$' case of the seventeenth equation in the list of 50 has the simpler form
\begin{equation} \label{XVII} 
\overset{\mdot \mdot}{w} = \frac{\overset{\mdot}{w}^2}{2 w}. \tag{{$\bf XVII$}}
\end{equation}
We may approach this equation by first rearranging it as $2 \overset{\mdot \mdot}{w}/\overset{\mdot}{w} = \overset{\mdot}{w}/w$ and so deducing the first integral $\overset{\mdot}{w}^2 = K w$ with an elementary solution for $w$; alternatively, we may rearrange the equation as $2 (\overset{\mdot}{w}/w)^{\mdot} = - (\overset{\mdot}{w}/w)^2$ and proceed accordingly. Instead, we may pass directly to the associated third-order equation: thus, rearrangement and differentiation again lead to 
$$2 w \: \overset{\mdot \mdot \mdot}{w} = 0$$
and thence to the quadratic $w = a z^2 + b z + c$; satisfaction of ${\bf XVII}$ requires that $b^2 - 4 a c = 0$ whence the quadratic is a perfect square. Again, passage to third order has proved to be  practically efficient. 

\medbreak 

As a last illustration of the third-order approach, we take the twenty-ninth equation in the list of 50: 
\begin{equation} \label{XXIX} 
\overset{\mdot \mdot}{w} = \frac{\overset{\mdot}{w}^2}{2 w} + \frac{3}{2} w^3. \tag{{$\bf XXIX$}}
\end{equation}
In this case, the rearrangement $2 w \overset{\mdot \mdot}{w} = \overset{\mdot}{w}^2 + 3 w^4$ leads after cancellation to the third-order equation 
$$\overset{\mdot \mdot \mdot}{w} = 6 w^2 \overset{\mdot}{w} = (2 w^3)^{\mdot}$$
from which we deduce that 
$$\overset{\mdot \mdot} {w} = 2 w^3 + k;$$
multiplication by $2 \overset{\mdot}{w}$ throughout generates the first integral 
$$\overset{\mdot}{w}^2 = w^4 + K w + L$$
and filtration through ${\bf XXIX}$ shows that the constant $L$ is zero. 

\bigbreak 

For the handling of some equations, an alternative identity may rival 
$$\big(2 w \overset{\mdot \mdot}{w} - \overset{\mdot}{w}^2\big)^{\mdot} = 2 w \: \overset{\mdot \mdot \mdot}{w}.$$
Explicitly, notice that 
$$\Big( \frac{\overset{\mdot}{w}^2}{w} \Big)^{\mdot} = \frac{2 \overset{\mdot}{w} \overset{\mdot \mdot}{w}}{w} - \frac{\overset{\mdot}{w}^3}{w^2} = 2 \frac{\overset{\mdot}{w}}{w} \Big( \overset{\mdot \mdot}{w} - \frac{\overset{\mdot}{w}^2}{2 w} \Big)$$
or 
$$\Big( \frac{\overset{\mdot}{w}^2}{w} \Big)^{\mdot} =  \frac{\overset{\mdot}{w}}{w^2} \big( 2 w \overset{\mdot \mdot}{w} - \overset{\mdot}{w}^2 \big).$$

\bigbreak 

For example, application of this identity to ${\bf XXIX}$ leads promptly to 
$$\Big( \frac{\overset{\mdot}{w}^2}{w} \Big)^{\mdot} = 3 w^2 \overset{\mdot}{w} = (w^3)^{\mdot}$$
and so to the first integral 
$$\overset{\mdot}{w}^2 = w^4 + K w.$$

\medbreak 

While on the topic of this rival identity, we shall demonstrate its usefulness in solving equation ${\bf XXXII}$ with which we began this section. Substitution of ${\bf XXXII}$ into the rival identity gives 
$$\Big( \frac{\overset{\mdot}{w}^2}{w} \Big)^{\mdot} = 2 \frac{\overset{\mdot}{w}}{w} \Big(- \frac{1}{2 w} \Big) = - \frac{\overset{\mdot}{w}}{w^2} = \Big( \frac{1}{w} \Big)^{\mdot}$$
whence (with a judicious naming of the constant of integration) 
$$\overset{\mdot}{w}^2 = 1 + 4 a w$$
and we have arrived at a first integral. Rather than press on and integrate one last time we step back and take a further derivative, following the same apparently perverse inclination that led us to pass from the fourth Painlev\'e equation to third order: thus 
$$2 \overset{\mdot}{w} \overset{\mdot \mdot}{w} = 4 a \overset{\mdot}{w}$$
so that shortly $\overset{\mdot \mdot}{w} = 2 a$ and $w = a z^2 + b z + c$ (with $b^2 - 4 a c = 1$ as before). 

\medbreak 

We remark that this same rival identity is also effective in handling ${\bf XVII}$, ${\bf XVIII}$ and ${\bf XIX}$, along with other equations in the list of 50, to the enjoyment of which we leave the reader. 

\bigbreak

\begin{center} 
{\small R}{\footnotesize EFERENCES}
\end{center} 
\medbreak 

[1] E.L. Ince. {\it Ordinary Differential Equations}, Longman, Green and Company (1926); Dover Publications (1956). 

\medbreak 

[2] P.L. Robinson, {\it Painlev\'e IV: roots and zeros}, arXiv 1609.02494 (2016). 

\medbreak 

[3] P.L. Robinson, {\it Painlev\'e II transcendents and their squares}, arXiv 1609.07166 (2016).

\end{document}